\documentclass{article}

\usepackage{hyperref}
\usepackage{nomencl}
\usepackage{chngcntr}
\usepackage{amsmath}
\usepackage{amsthm}
\usepackage{amscd}
\usepackage{bbm}
\usepackage{enumerate}
\usepackage{amssymb}
\usepackage{palatino}
\usepackage{amscd}
\usepackage{verbatim}
\usepackage{mathrsfs}
\usepackage{esint}
\usepackage{xcolor}

\usepackage[T2A,T1]{fontenc}
\usepackage[utf8]{inputenc}

\DeclareSymbolFont{cyrillic}{T2A}{cmr}{m}{n}
\def\makecyrsymbol#1#2{%
  \begingroup\edef\temp{\endgroup
    \noexpand\DeclareMathSymbol{\noexpand#1}
    {\noexpand\mathalpha}{cyrillic}%
    {\expandafter\expandafter\expandafter
     \calccyr\expandafter\meaning\csname T2A\string#2\endcsname\end}}%
  \temp}
\expandafter\def\expandafter\calccyr\string\char#1\end{#1}

\makecyrsymbol\hardsign\cyrhrdsn
\makecyrsymbol\bighardsign\CYRHRDSN
\makecyrsymbol\iy\cyrery
\makecyrsymbol\bigiy\CYRERY
\makecyrsymbol\softsign\cyrsftsn
\makecyrsymbol\bigsoftsign\CYRSFTSN

\newtheoremstyle{dotless}{}{}{\itshape}{}{\bfseries}{}{ }{}

\newtheorem{thm}{Theorem}

\newtheorem{conj}[thm]{Conjecture}

\theoremstyle{dotless}

\counterwithin{equation}{section}
\counterwithin{thm}{section}

\newcommand{\rank}{\mathop{\mathrm{rank}}}

\newcommand{\Q}{\mathbb{Q}}

\newcommand{\Nm}{\mathrm{Nm}}

\newcommand{\F}{\mathbb{F}}

\newcommand{\Gal}{\mathop{\mathrm{Gal}}}

\newcommand{\pfrak}{\mathfrak{p}}

\newcommand{\inj}{\hookrightarrow}

\newcommand{\Ind}{\mathrm{Ind}}

\newcommand{\eps}{\varepsilon}

\newcommand{\afrak}{\mathfrak{a}}

\newcommand{\ofrak}{\mathfrak{o}}

\newcommand{\G}{\mathbb{G}}

\renewcommand{\o}{\ofrak}

\newcommand{\Cl}{\mathrm{Cl}}

\newcommand{\Qbar}{\overline{\Q}}

\newcommand{\Jac}{\mathrm{Jac}}

\newcommand{\aff}{\mathrm{aff.}}
\newcommand{\Pic}{\mathrm{Pic}}

\let\phi\varphi

\makeatletter
\let\@@pmod\pmod
\DeclareRobustCommand{\pmod}{\@ifstar\@pmods\@@pmod}
\def\@pmods#1{\mkern4mu({\operator@font mod}\mkern 6mu#1)}
\makeatother

 \DeclareFontFamily{U}{wncy}{}
    \DeclareFontShape{U}{wncy}{m}{n}{<->wncyr10}{}
    \DeclareSymbolFont{mcy}{U}{wncy}{m}{n}
    \DeclareMathSymbol{\Sha}{\mathord}{mcy}{"58}

\title{\vspace{-0.15in}Note on a theorem of Professor X.}

\author{Levent Alp\"{o}ge}

\date{}

\begin{document}
\maketitle
\begin{abstract}
Between his arrival in Frankfurt in $1922$ and and his proof of his famous finiteness theorem for integral points in $1929$, Siegel had no publications. He did, however, write a letter to Mordell in $1926$ in which he explained a proof of the finiteness of integral points on hyperelliptic curves. Recognizing the importance of this argument (and Siegel's views on publication), Mordell sent the relevant extract to be published under the pseudonym "X".

The purpose of this note is to explain how to optimize Siegel's $1926$ technique to obtain the following bound. Let $K$ be a number field, $S$ a finite set of places of $K$, and $f\in \o_{K,S}[t]$ monic of degree $d\geq 5$ with discriminant $\Delta_f\in \o_{K,S}^\times$. Then: $$\#|\{(x,y) : x,y\in \o_{K,S}, y^2 = f(x)\}|\leq  2^{\rank{\Jac(C_f)(K)}}\cdot O(1)^{d^3\cdot ([K:\Q] + \#|S|)}.$$

This improves bounds of Evertse-Silverman and Bombieri-Gubler from $1986$ and $2006$, respectively.

The main point underlying our improvement is that, informally speaking, we insist on "executing the descents in the presence of only one root (and not three) until the last possible moment".
\end{abstract}

\section{Introduction.}
The technique introduced in Siegel's $1926$ letter to Mordell \cite{siegel-x} to prove the finiteness of integral points on hyperelliptic curves\footnote{By this common and incorrect abbreviation we really mean integral solutions of $y^2 = f(x)$. Siegel's $1926$ proof does not control integral solutions of e.g.\ $y^3 + x\cdot y = x^4$, or, said another way, integral points with respect to an effective divisor containing no nonzero divisor symmetric under the hyperelliptic involution. Baker effectivized Siegel's $1926$ argument and thus gave an effective finiteness proof in the $y^2 = f(x)$ cases, but he did not effectivize Siegel's $1929$ finiteness theorem in the case of hyperelliptic curves.} can be summarized as: $2$-descent on the curve, then $3$-descent on $\G_m$. From this summary the intuitive guess for the bound that the argument "should" produce is of course something of shape $2^{\rank}\cdot 3^{\text{$\#$ of prime factors of the discriminant}}$. However a consultation of the literature yields worse estimates.\footnote{To our knowledge the best bounds in the literature are the $2006$ estimate of Bombieri-Gubler \cite[Theorem $5.3.5$]{bombieri-gubler} and the $1986$ estimate of Evertse-Silverman \cite{evertse-silverman}. In both bounds the "expected" $2^{\rank}$ factor is replaced by a power of the size of the $2$-part of the class group of an extension gotten by adjoining three Weierstrass points.}

In this note we rectify the situation.

\begin{thm}\label{the main theorem}
Let $K/\Q$ be a number field. Let $S$ be a finite set of places of $K$. Let $f\in \o_{K,S}[t]$ be monic of degree $d\geq 5$ with discriminant $\Delta_f\in \o_{K,S}^\times$. Let $\mathcal{C}_f^\aff : y^2 = f(x)$ be the Weierstrass model of the hyperelliptic curve $C_f$ corresponding to $f$. Then: $$\#|\mathcal{C}_f^\aff(\o_{K,S})|\leq 2^{\rank{\Jac(C_f)(K)}}\cdot O(1)^{d^3\cdot ([K:\Q] + \#|S|)}.$$

Also, for all $K$-irreducible $g_i\in \o_{K,S}[t]$ with $\sum_i \deg{g_i}\geq 3$ and $\prod_i g_i\big\vert f$, $$\#|\mathcal{C}_f^\aff(\o_{K,S})|\leq \left(\prod_i \#|\Cl(\o_{K_{g_i},S})[2]|\right)\cdot O(1)^{d^3\cdot ([K:\Q] + \#|S|)},$$ where $K_{g_i} := K[t]/(g_i)$ and $\o_{K_{g_i},S} := \o_{K_{g_i}}\otimes_{\o_K} \o_{K,S}$.
\end{thm}\noindent
We include the second statement in case e.g.\ the product of the linear and quadratic irreducible factors of $f$ is of degree at least $3$.

In the elliptic curve case an even stronger bound is available because one can execute the whole descent over $K$ (following Mordell) \cite{paper-with-wei}\footnote{See also Chapter $2$ of \cite{my-thesis}, which is based on the same work.}. In the superelliptic ($y^m = f(x), m > 2$) case as usual one does not need to execute a $3$-descent on $\G_m$ so the bound also improves.

Let us explain how the argument goes. First, the $2$-descent. \`{A} la Fermat, given $y^2 = f(x)$, adjoin a root $\rho$ and conclude that $x - \rho$ is almost a square, say $x-\rho = \alpha\cdot \beta^2$ with $\alpha,\beta\in K[\rho] := K[t]/(f)$.

Now adjoin three roots. Let $L/K$ be an extension containing three roots, say $\rho_1, \rho_2, \rho_3$. Then, for all $K$-maps $\sigma: K[\rho]\to L$, we find $x - \sigma(\rho) = \sigma(\alpha)\cdot \sigma(\beta)^2$. Thus (obvious notation) $\rho_j - \rho_i = \alpha_i\cdot \beta_i^2 - \alpha_j\cdot \beta_j^2$.

Now pass to $L' := L(\sqrt{\alpha_1}, \sqrt{\alpha_2}, \sqrt{\alpha_3})$. We obtain the six elements $\gamma_{ij,\pm} := \beta_i\sqrt{\alpha_i}\pm \beta_j\sqrt{\alpha_j}$, all divisors of $\rho_j - \rho_i$. By hypothesis they are therefore all $S'$-units, with $S'$ the set of places of $L'$ above a place in $S$.

Now for the $3$-descent. We may then write $\gamma_{ij,\pm} = \delta_{ij,\pm}\cdot \eps_{ij,\pm}^3$ with $\delta_{ij,\pm}$ representatives for $\o_{L',S'}^\times/3$. The relation $\gamma_{12,+} - \gamma_{23,+} = \gamma_{13,-}$ becomes the cubic Thue equation $\frac{\delta_{12,+}}{\delta_{13,-}}\cdot \left(\frac{\eps_{12,+}}{\eps_{13,-}}\right)^3 - \frac{\delta_{23,+}}{\delta_{13,-}}\cdot \left(\frac{\eps_{23,+}}{\eps_{13,-}}\right)^3 = 1$.

We conclude with a bound of Evertse.

As a last remark: note that $2^{\rank{\Jac\,{(C_f)}(K)}}$ is $\ll 1$ on average (ordering as usual by height) \cite{bhargava-gross, shankar-wang}. This statement is why we tried for the stronger bound \cite{paper-with-wei} in the elliptic curve case, after which we questioned why a similar bound did not exist in the hyperelliptic curve case.

\section{Acknowledgments.}

This note is based on Chapter $4$ of the author's Ph.D.\ thesis at Princeton University. I would like to thank both my advisor Manjul Bhargava and Peter Sarnak for their patience and encouragement. I would also like to thank Jacob Tsimerman and Nina Zubrilina for informative discussions. Finally I thank the National Science Foundation (via their grant DMS-$2002109$), Columbia University, and the Society of Fellows for their support during the pandemic.

\section{Proof of Theorem \ref{the main theorem}.}

\begin{proof}[Proof of Theorem \ref{the main theorem}]
Let $g\in \o_{K,S}[t]$ with $g\vert f$ be a monic divisor of $f$ with $\deg{g}\geq 3$ (for the first part of the theorem we will take $g = f$). Write $f =: g h$ with $h\in \o_{K,S}[t]$. Let $L/K$ be an extension of minimal degree containing at least three roots of $g$.

Factorize $g =: \prod_i g_i$ into $K$-irreducible factors $g_i\in \o_{K,S}[t]$. Note that $(g_i, g_j) = (1)$ for each $i\neq j$ because $\Delta_f\in \o_{K,S}^\times$. Thus $K_g\simeq \bigoplus_i K_{g_i}$, with $K_{g_i} := K[t]/(g_i)$ and $K_g := K[t]/(g)$. We will repeatedly write $\o_{K_{g_i}, S}$ etc.\ for the evident localizations, e.g.\ $\o_{K_{g_i},S} := \o_{K_{g_i}}\otimes_{\o_K} \o_{K,S}$. Write $\rho^{(i)}$ for the image of $t$ in $K_{g_i} = K[t]/(g_i)$.

By the Chebotarev density theorem (and its explicit error term) applied to the Hilbert class field of $K_{g_i}$, each ideal class of $K_{g_i}$ contains a prime of norm $\ll |\Delta_{K_{g_i}}|^{O(1)}$. (We will only use existence.) Let then $P^{(i)}$ be a minimal set of prime representatives of $\Cl(\o_{K_{g_i},S})[2]$ in $K_{g_i}$ of norm $\ll |\Delta_{K_{g_i}}|^{O(1)}$.

Now let us begin the argument. For $(x,y)\in \mathcal{C}_f^\aff(\o_{K,S})$, $(y)^2 = (x-\rho^{(i)})\cdot \left(\frac{g_i(x)}{x-\rho^{(i)}}\right)\cdot \prod_{j\neq i} (g_j(x))\cdot (h(x))$ as ideals in $\o_{K_{g_i},S}$. Without loss of generality $f(x)\neq 0$. Since $\Delta_f\in \o_{K,S}^\times$, it follows that there is an ideal $\afrak_i\subseteq \o_{K_{g_i},S}$ with $(x-\rho^{(i)}) = \afrak_i^2$. Thus there is a $\pfrak_i\in P^{(i)}$ with $\pfrak_i\equiv \afrak_i$ modulo principal ideals of $K_{g_i}$. Hence $\afrak_i$ is principal in $\o_{K_{g_i},S\cup \{\pfrak_i\}}$ (obvious meaning).

Let $\alpha_i\in \o_{K_{g_i},S\cup \{\pfrak_i\}}$ be such that $\afrak_i = (\alpha_i)$ as ideals of $\o_{K_{g_i},S\cup \{\pfrak_i\}}$. Thus $x - \rho^{(i)} = \alpha_i^2\cdot (\in \o_{K_{g_i}, S\cup \{\pfrak_i\}}^\times)$, where $(\in \o_{K_{g_i}, S\cup \{\pfrak_i\}}^\times)$ denotes an element of the $(S\cup \{\pfrak_i\})$-units of $K_{g_i}$.

Let $U^{(i)}$ be a minimal set of representatives of $\o_{K_{g_i},S\cup \{\pfrak_i\}}^\times/2$ (that is, modulo squares). It follows that there are $\gamma_i\in U^{(i)}$ and $u^{(i)}\in \o_{K_{g_i},S\cup \{\pfrak_i\}}^\times$ such that $x - \rho^{(i)} = \gamma_i\cdot (\alpha_i\cdot u^{(i)})^2$. Let $\eta_i := \alpha_i\cdot u^{(i)}$. Thus $x - \rho^{(i)} = \gamma_i\cdot \eta_i^2$.

Therefore we find that, for each $K$-embedding $\sigma: K_{g_i}\inj L$, we have $x - \sigma(\rho^{(i)}) = \sigma(\gamma_i)\cdot \sigma(\eta_i)^2$.

By definition of $L$, for each $1\leq k\leq 3$ there is an $i_k$ and a $K$-embedding $\tau_k: K_{g_{i_k}}\inj L$ such that the corresponding roots $\tau_k(\rho^{(i_k)})$ are pairwise distinct. Let then:
\begin{align*}
\kappa_k &:= \tau_k(\rho^{(i_k)}),
\\\lambda_k &:= \tau_k(\gamma_{i_k}),
\\\mu_k &:= \tau_k(\eta_{i_k}).
\end{align*}
Hence $x - \kappa_k = \lambda_k\cdot \mu_k^2$ as elements of $L$. Hence for $k\neq \ell$ we find that $\kappa_\ell - \kappa_k = \lambda_k\cdot \mu_k^2 - \lambda_\ell\cdot \mu_\ell^2$.

We are done with the $2$-descent on $C_f$. Now it is time for the $3$-descent on $\G_m$. Let $L_{k\ell} := L\left(\sqrt{\lambda_k}, \sqrt{\lambda_\ell}\right)$. Thus $(1) = (\kappa_\ell - \kappa_k) = \left(\sqrt{\lambda_k} \mu_k - \sqrt{\lambda_\ell} \mu_\ell\right)\cdot \left(\sqrt{\lambda_k} \mu_k + \sqrt{\lambda_\ell} \mu_\ell\right)$ as ideals of $\o_{L_{k\ell},S\cup \{\pfrak_k, \pfrak_\ell\}}$ (obvious meaning), the first equality following from $\Delta_f\in \o_{K,S}^\times$.

Let $V_{k\ell}$ be a minimal set of representatives of $\o_{L_{k\ell}, S\cup \{\pfrak_k, \pfrak_\ell\}}^\times / 3$ (that is, modulo cubes). It follows that there are $v_{k\ell,\pm}\in V_{k\ell}, \zeta_{k\ell,\pm}\in \o_{L_{k\ell}, S\cup \{\pfrak_k, \pfrak_\ell\}}^\times$ for which $\sqrt{\lambda_k} \mu_k \pm \sqrt{\lambda_\ell} \mu_\ell = v_{k\ell,\pm}\cdot \zeta_{k\ell,\pm}^3$. Hence:
\begin{align*}
0 &= (\sqrt{\lambda_1} \mu_1\pm \sqrt{\lambda_2} \mu_2) \mp (\sqrt{\lambda_2} \mu_2 \pm \sqrt{\lambda_3} \mu_3) - (\sqrt{\lambda_1} \mu_1 - \sqrt{\lambda_3} \mu_3)
\\&= v_{12,\pm}\cdot \zeta_{12,\pm}^3 \mp v_{23,\pm}\cdot \zeta_{23,\pm}^3 - v_{13,-}\cdot \zeta_{13,-}^3.
\end{align*}
We may of course rearrange this as $\frac{v_{12,\pm}}{v_{13,-}}\cdot \left(\frac{\zeta_{12,\pm}}{\zeta_{13,-}}\right)^3 \mp \frac{v_{23,\pm}}{v_{13,-}}\cdot \left(\frac{\zeta_{23,\pm}}{\zeta_{13,-}}\right)^3 = 1$.

Now for Evertse's bound. Let $F_\pm(X,Y) := \frac{v_{12,\pm}}{v_{13,-}}\cdot X^3\mp \frac{v_{23,\pm}}{v_{13,-}}\cdot Y^3\in \o_{L_{123}}[X,Y]$, with $L_{123} := L(\sqrt{\lambda_1}, \sqrt{\lambda_2}, \sqrt{\lambda_3})$. Then $F_\pm$ is a cubic form with nonzero discriminant. Moreover if $F_\pm(X,Y) = F_\pm(tX,tY) = 1$ then $t^3 = 1$. Therefore by \cite{evertse} the number of solutions of $F_\pm(X,Y) = 1$ with $(X,Y)\in \o_{L_{123}, S\cup \{\pfrak_1, \pfrak_2, \pfrak_3\}}$ is $\ll O(1)^{[L:\Q] + [L : K]\cdot \#|S|}$.

We are done with the argument. Now let us examine the tally. Let us first show that our original point $(x,y)$ can be recovered from the pairs $\left(\frac{\zeta_{12,\pm}}{\zeta_{13,-}}, \frac{\zeta_{23,\pm}}{\zeta_{13,-}}\right)$ up to $O(1)$ many possibilities. To see this first multiply the two $X$-coordinates together to form $\frac{\zeta_{12,+}\cdot \zeta_{12,-}}{\zeta_{13,-}^2}$. Note that:
\begin{align*}
v_{12,+}\cdot v_{12,-}\cdot \zeta_{12,+}^3\cdot \zeta_{12,-}^3 &= (\sqrt{\lambda_1} \mu_1 + \sqrt{\lambda_2} \mu_2)\cdot (\sqrt{\lambda_1} \mu_1 - \sqrt{\lambda_2} \mu_2)
\\&= \lambda_1 \mu_1^2 - \lambda_2 \mu_2^2
\\&= \kappa_2 - \kappa_1.
\end{align*}

Thus it follows that $\zeta_{12,+}^3\cdot \zeta_{12,-}^3 = \frac{\kappa_2 - \kappa_1}{v_{12,+}\cdot v_{12,-}}$. Hence the cube of the product of the two $X$-coordinates is $\zeta_{13,-}^{-6}\cdot \left(\frac{\kappa_2 - \kappa_1}{v_{12,+}\cdot v_{12,-}}\right)$. We note that the term in parentheses is fixed (in terms of our choices up til now). Thus we may recover $\zeta_{13,-}^6$, and hence $\zeta_{13,-}$ up to at most six choices. Having done so we return to the $X$-coordinates of both solutions and recover $\zeta_{12,+}$ and $\zeta_{12,-}$. Then the equality \begin{align*}
2\sqrt{\lambda_1} \mu_1 &= (\sqrt{\lambda_1} \mu_1 + \sqrt{\lambda_2} \mu_2) + (\sqrt{\lambda_1} \mu_1 - \sqrt{\lambda_2} \mu_2)
\\&= (v_{12,+}\cdot \zeta_{12,+}^3) + (v_{12,-}\cdot \zeta_{12,-}^3)
\end{align*}
implies that we can recover $2\sqrt{\lambda_1} \mu_1$. Squaring this we find that we can recover $4\lambda_1 \mu_1^2 = 4(x-\kappa_1)$. Since $\kappa_1$ is fixed we can recover $x$, and then there are at most two choices for $y$ given $x$, so we can indeed recover the point up to $O(1)$ many choices.

So we see that a point $(x,y)\in \mathcal{C}_f^\aff(\o_{K,S})$ is determined up to $O(d^3)$ (arising from the choice of three roots in $L$) many choices by the data $$(\pfrak_1, \pfrak_2, \pfrak_3, \gamma_1, \gamma_2, \gamma_3, v_{12,+}, v_{12,-}, v_{23,+}, v_{23,-}, v_{31,-}).$$ The number of choices for each $\pfrak_i$ is $\#|P^{(i)}| = \#|\Cl(\o_{K_{g_i},S})[2]|$. The number of choices for each $\gamma_i\in U^{(i)}$ is $\#|U^{(i)}|\ll 2^{\deg{g_i}\cdot (\#|S| + [K:\Q])}$ by Dirichlet. Similarly the number of choices for each $v_{k\ell,\pm}\in V_{k\ell}$ is $\#|V_{k\ell}|\ll 3^{4\cdot [L : \Q] + 4\cdot [L : K]\cdot \#|S|}$. Of course $[L:K]\leq d^3$.

Therefore the total number of tuples $(\pfrak_1, \pfrak_2, \pfrak_3, \gamma_1, \gamma_2, \gamma_3, v_{12,+}, v_{12,-}, v_{23,+}, v_{23,-}, v_{31,-})$ is $$\ll \left(\prod_i \#|\Cl(\o_{K_{g_i},S})[2]|\right)\cdot O(1)^{d^3\cdot ([K:\Q] + \#|S|)}.$$ We are done with the second part of the theorem.

So take $g = f$, and let us count the ideal classes of $\bigoplus_i \o_{K_{g_i},S}\simeq \o_{K_f,S}$ that could possibly arise in the $2$-descent step (we have bounded said count by $\leq \prod_i \#|\Cl(\o_{K_{g_i},S})[2]|$ and we claim it is also $\ll 2^{\rank{\Jac\,{(C_f)}(K)}}\cdot O(1)^d$).

Let $J_f := \Jac\,{C_f}$. Let $W\subseteq C_f(\Qbar)$ be the set of Weierstrass points of $C_f$. Let $\infty\in W$ be a point at infinity. We embed $C_f\inj J_f = \Pic^0(C_f)$ via $P\mapsto P - \infty$. Thus as $\Gal(\Qbar/K)$-modules $J_f[2]\simeq \F_2[W - \{\infty\}] / \F_2\cdot \left(\sum_{P\in W - \{\infty\}} P\right)$. In other words, $\Ind_K^{K_f} \F_2\simeq \F_2\oplus J_f[2]$, where we have written $\Ind_K^{K_f}(\bullet) := \bigoplus_i \Ind_K^{K_{g_i}}(\bullet)$. Thus $H^1(K_f, \F_2)\simeq H^1(K, \Ind_K^{K_f} \F_2)\simeq H^1(K, \F_2)\oplus H^1(K, J_f[2])$, where the first isomorphism follows by Shapiro's lemma (and $H^1(K_f, \bullet) := \bigoplus_i H^1(K_{g_i}, \bullet)$). By Kummer it follows that $H^1(K, J_f[2])\simeq (K_f^\times / 2)_{\Nm = \square}$.

Thus by taking invariants of $0\to J_f[2]\to J_f\to J_f\to 0$ we obtain $J_f(K)/2\inj H^1(K, J_f[2])\simeq (K_f^\times/2)_{\Nm = \square}$. Write $G\subseteq H^1(K, J_f[2])$ for the image of this map. Note that the restriction of this map to $C_f(K) - W$ is simply $(x,y)\mapsto x-\rho$, so similarly write $G'\subseteq G$ for the image of $\mathcal{C}^\aff_f(\o_{K,S}) - W$.

We have already seen that each $g\in G'\subseteq G\inj (K_f^\times/2)_{\Nm = \square}$ gives rise to a class represented by an $\alpha\in K_f^\times$ for which $v_\pfrak(\alpha)$ is even for all primes $\pfrak\subseteq \o_{K_f,S}$ (because $\Delta_f\in \o_{K,S}^\times$). From such a class we may produce an element of $\Cl(\o_{K_f,S})[2]$ via $\alpha\mapsto \afrak$ such that $(\alpha) = \afrak^2$ as ideals of $\o_{K_f,S}$.

The corresponding map $\mathcal{C}^\aff_f(\o_{K,S}) - W\to \Cl(\o_{K_f,S})[2]$ is precisely (up to our choice of representatives) our map $(x,y)\mapsto (\pfrak_i)_i$. It therefore suffices to show that $\#|G'|\ll 2^{\rank{J_f(K)}}\cdot O(1)^d$. However $\#|G| = \#|(J_f[2])(K)|\cdot 2^{\rank{J_f(K)}}$ and $G'\subseteq G$.
\end{proof}

\renewcommand{\refname}{References.}

\bibliographystyle{amsplain}

\bibliography{noteonatheoremofprofessorx}

\end{document}